\documentclass[11pt, epsfig]{article}
\usepackage{epsfig, amsmath, amssymb, amsthm, times}
\usepackage{t1enc}
\def\id{\mathcal J}

\def\R{\mathbb{R}}

\def\C{{\mathbb C}}
\def\N{{\mathbb N}}

\def\<{\langle}
\def\>{\rangle}
\newcommand{\bel}{\begin{equation}\label}
\newcommand{\ee}{\end{equation}}
      \newtheorem{theorem}{Theorem}
       \newtheorem{proposition}[theorem]{Proposition}
       
       \newtheorem{lemma}[theorem]{Lemma}
       \newtheorem{remark}{Remark}
\theoremstyle{definition}

\title{The Lukacs theorem and the Olkin-Baker equation}
\author{Roman Ger\thanks{Instytut Matematyki, Uniwesytet \'Sl\k{a}ski, Katowice, Poland,\newline e-mail: roman.ger@us.edu.pl},\hspace{1mm} Jolanta Misiewicz\thanks{Wydzia{\l} Matematyki i Nauk Informacyjnych, Politechnika Warszawska, Warszawa, Poland,\newline e-mail: J.Misiewicz@mini.pw.edu.pl},\hspace{1mm} Jacek Weso\l owski\thanks{Wydzia{\l} Matematyki i Nauk Informacyjnych, Politechnika Warszawska, Warszawa, Poland,\newline e-mail: wesolo@mini.pw.edu.pl}}

\begin{document}
\maketitle
\begin{abstract}
The Olkin-Baker functional equation, except of being studied inside the theory of functional equations, is closely related to the celebrated Lukacs characterization of the gamma distribution. Its deeper understanding in the case of measurable unknown functions is essential to settle a challenging question of multivariate extensions of the Lukacs theorem. In this paper, first, we provide a new approach to the additive Olkin-Baker equation which holds almost everywhere on $(0,\infty)^2$ (with respect to the Lebesgue measure on $\R^2$) under measurability assumption. Second, this new approach is adapted to the case when unknown functions are allowed to be non-measurable and the complete solution is given in such a general case. Third, the Olkin-Baker equation holding outside of a set from proper linearly invariant ideal of subsets of $\R^2$ is considered.
\end{abstract}

\section{Introduction}
One of the most classical results of characterizations of
probability distributions is the Lukacs theorem, which states that
{\em if $X$ and $Y$ are positive, non-degenerate and independent
random variables such that $U=X+Y$ and $V=X/(X+Y)$ are also
independent then $X$ and $Y$ have gamma distributions with the
same scale and possibly different shape parameters}.

The proof given in Lukacs (1955) exploits the approach through a
differential equation for the Laplace transforms. This technique
was successfully developed for matrix variate versions of the
Lukacs theorem in Olkin and Rubin (1962) and Casalis and Letac
(1996), where the Wishart distribution was characterized through
independence of $U=X+Y$ and $V=w(X+Y)\:X\:w^T(X+Y)$, where $w$ is
so called division algorithm, that is $w(a)aw^T(a)=I$ for any
positive definite matrix $a$. However an additional strong
assumption of invariance of the distribution of $V$ by a group of
authomorphisms of the cone of positive definite matrices was
additionally imposed. To avoid this restrictive invariance
condition Bobecka and Weso\l owski (2002) designed a new approach
to Lukacs characterization based on densities. Assuming that the
densities are strictly positive on the cone of positive definite
matrices and that they are twice differentiable they proved the
characterization of Wishart distributions through independence of
$U=X+Y$ and $V=(X+Y)^{-1/2}X(X+Y)^{-1/2}$, where $a^{1/2}$ denotes
the unique symmetric root of a positive definite matrix. The proof
was based on solutions of two functional equations for real
functions defined on the cone of positive definite matrices.
Exploiting the same technique, Hassairi, Lajmi and Zine (2008)
with the same technical assumptions on densities proved that
independence of $U=X+Y$ and $V=[W(X+Y)]^{-1}X[W^T(X+Y)]^{-1}$,
where $W(X+Y)\:W^T(X+Y)=X+Y$ is the Cholesky decomposition of $X+Y$,
that is $W(X+Y)$ is an upper triangular matrix, characterizes a
wider family of distributions called Riesz-Wishart (further
development in the case of this division algorithm, still under
twice differentiable densities, was obtained for homogenous cones
in Boutoria (2005, 2009) and Boutoria, Hassairi, Massam (2011)). This fact shows that the invariance property
assumed in Olkin and Rubin (1962) and Casalis and Letac (1996) was rather
not of technical character. It appears that the Lukacs
independence condition may define different "gamma" distributions
depending on the division algorithm used for designing the
variable $V$.

It is somewhat disappointing that these results depend so much
on smoothness conditions for densities. Even if the condition of
existence of strictly positive densities can be easier accepted,
twice differentiability of the densities seems to be too much.
Therefore it is of great interest to get rid of these technical
restrictions if it is possible. Moving in this direction, the first
of the functional equations from Bobecka and Weso\l owski (2002),
the equation of ratios, has been recently solved in Weso\l owski
(2007) for functions defined on the Lorentz cone. Studying the Lukacs theorem on the Lorentz cone ${\mathcal
V}=\{(x_0,x_1,\ldots,x_n)\in\R^{n+1}:\:x_0>\sqrt{x_1^2+\ldots+x_n^2}\}$ is the right approach,
if one wants to extend the Lukacs result to the $\R^{n+1}$
setting. Another trial in this direction given in Bobecka and
Weso\l owski (2004) through a coordinate-wise version of
the Lukacs independence condition led to a very special
distribution of independent sub-vectors, components of which are
scaled versions of a univariate gamma variable.

Therefore, there is a good reason for looking for new proofs of the classical Lukacs result in univariate case. Except of being of interest on its own it may give a new insight into
what can be done in the multivariate setting, as explained above. The basic functional equation related to this issue, when one considers relations which have to hold for  densities, is the celebrated Olkin-Baker equation
\bel{ob}
f(x)g(y)=p(x+y)q(x/y)
\ee
with unknown functions $f,\,g,\,p,\,q:(0,\infty)\to(0,\infty)$. The problem of solving this equation was posed by Olkin (1975). Its general solution under the assumption that it holds for all $(x,y)\in(0,\infty)^2$ was given by Baker (1976). The proof was based on two additional lemmas and was quite complicated. The equation \eqref{ob} was also analyzed in Lajko (1979), who applied the approach developed in Dar\'oczy, Lajk\'o and Sz\'ekelyhidi (1979) allowing to use known methods for the Jensen functional equation. Recently, M\'esz\'aros (2010) solved this equation assuming that it is satisfied $l_2$-almost everywhere on $(0,\infty)^2$. (Throughout this paper $l_n$ denotes the Lebesgue measure in $\R^n$). Her approach was based on J\'arai's regularization technique (see J\'arai, 1995 and 2005), which actually allowed to reduce the problem to the Olkin-Baker equation with unknown continuous functions for which the equation \eqref{ob} holds everywhere. Consequently, the final result followed directly from the original Baker solution. More recently, Lajk\'o and M\'esz\'aros (2012), using another method developed by J\'arai (see also J\'arai, 2005), showed that it suffices to assume that unknown nonnegative functions in \eqref{ob} are positive on some sets of positive Lebesgue measure.

The motivations for this paper are two-fold: coming from probability - we seek new approach to the Olkin-Baker equation which may lead to the matrix-variate version of the Lukacs theorem for the Wishart distribution; coming from the functional equations theory - we seek general solution of the Olkin-Baker equation holding almost everywhere, or even more generally, holding outside a set from proper linearly invariant ideal of subsets of $\R^2$. We develop a new method of solution of the equation \eqref{ob} holding $l_2$-almost everywhere in $(0,\infty)^2$ first under measurability assumption for unknown functions. Further this method is extended to cover the case of general solution of \eqref{ob} holding $l_2$-almost everywhere in $(0,\infty)^2$ in terms of additive and logarithmic type functions. In the course of the argument we introduce a notion of "semi-constant" function. Such a function $f:(0,\infty)\to\R$ satisfies $f(xy)=f(y)$ for $l_1$-almost every $x>0$ and any $y\in(0,\infty)\setminus E_x$, where $l_1(E_x)=0$. The rest of the paper is organized as follows: In Section 2 we explain the connection between the Lukacs characterization and the Olkin-Baker equation. Section 3 is devoted to a new approach to the Olkin-Baker equation under assumptions that unknown functions are measurable and the equation is satisfied $l_2$-almost everywhere on $(0,\infty)^2$. In Section 4, using the concept of semi-constant function we derive general solution of the Olkin-Baker equation when it holds $l_2$-almost everywhere in $(0,\infty)^2$ and no regularity assumptions on unknown functions are imposed. In Section 5 we show how to adopt the reasoning of previous cases to an abstract setting when the equation is  satisfied outside of a set belonging to proper linearly invariant ideal of subsets of $\R^2$.

\section{Lukacs theorem with densities}

The result we formulate below is a special case of the Lukacs
theorem and as such is well known. The main novelty is its rather elementary proof based on densities and on a new approach to the Olkin-Baker equation. This approach has recently proved to be useful in multivariate extensions - it has been used e.g. in the proof of the Lukacs theorem on the Lorentz cone in Ko\l odziejek (2010).

\begin{theorem}\label{luk}
Let $X$ and $Y$ be independent random variables having strictly
positive densities defined on $(0,\infty)$. If $U=X+Y$ and
$V=X/(X+Y)$ are also independent, then there exist positive
numbers $p,q,a$ such that $X\sim G(p,a)$ and $Y\sim G(q,a)$, where
$G(r,c)$ denotes the gamma distribution with the shape parameter
$r>0$ and the scale parameter $c>0$, which is defined by the
density
$$
f(x)=\frac{c^r}{\Gamma(r)}x^{r-1}e^{-cx}I_{(0,\infty)}(x)\;.
$$
\end{theorem}

\begin{proof}
It is standard to see that the independence condition can be
equivalently rewritten as follows: there exists a set $M\subset
(0,\infty)^2$ such that $l_2(M)=0$ and
$$
f_U(x+y)f_V\left(\frac{x}{x+y}\right)=(x+y)f_X(x)f_Y(y)\quad\forall\:(x,y)\in(0,\infty)^2\setminus M.
$$
Taking logarithms of both sides we arrive at
\begin{equation}\label{abcd}
a(x)+b(y)=c(x+y)+d(x/y)\quad\forall\:(x,y)\in(0,\infty)^2\setminus
M,
\end{equation}
where $a=\log(f_X)$, $b=\log(f_Y)$, $c(x)=\log(f_U(x))-\log(x)$ and $d(x)=\log[f_V(x/(1+x))]$, $x\in(0,\infty)$.

Now, using the result of Prop. \ref{eq1} below we get
$$
f_X\propto x^{\varkappa_1}e^{-\lambda
x}I_{(0,\infty)}(x)\quad\mbox{and}\quad f_Y(y)\propto
y^{\varkappa_2}e^{-\lambda y}I_{(0,\infty)}(y),
$$
where $\propto$ means equality up to a multiplicative constant.
Since $f_X$ and $f_Y$, as densities, are integrable on
$(0,\infty)$, we have that $p_i=\varkappa_i+1>0$ and $\lambda>0$.
\end{proof}

\section{Almost everywhere Olkin-Baker functional equation under measurability}

The main result of this section is the solution of the equation
(\ref{abcd}) holding $l_2$-almost everywhere in $(0,\infty)^2$ under measurability assumptions through a new method which is neither based on Jensen equation (as in Lajk\'o, 1979) nor on J\'arai's regularization techniques (as in M\'esz\'aros, 2010). In Section 4, this method will be extended to  the general Olkin-Baker equation holding $l_2$-almost everywhere in $(0,\infty)^2$ with no regularity assumptions whatsoever on the unknown functions.

\begin{proposition}\label{eq1}
Let $a$, $b$, $c$ and $d$ be real Borel measurable functions on
$(0,\infty)$. Assume that there exists a measurable set
$M\subset(0,\infty)^2$ such that $l_2(M)=0$ and
\begin{equation}\label{abcd1}
a(x)+b(y)=c(x+y)+d\left(\frac{x}{y}\right)\;,\;\;\;(x,y)\in(0,\infty)^2\setminus
M\;.
\end{equation}

Then there exist real constants $\lambda$, $\varkappa_1$, $\varkappa_2$,
$\alpha$, $\beta$, $\gamma$ and $\delta$  satisfying
$\alpha+\beta=\gamma+\delta$ such that for $l_1$-almost all
$x\in(0,\infty)$
$$
a(x)=\lambda
x+\varkappa_1\log(x)+\alpha\;,\;\;\;b(x)=\lambda\:x+\varkappa_2\log(x)+\beta,
$$
$$
c(x)=\lambda\:x+(\varkappa_1+\varkappa_2)\log(x)+\gamma\;,\;\;\;d(x)=\varkappa_1\log\left(\tfrac{x}{x+1}\right)-\varkappa_2\log(1+x)+\delta.
$$
\end{proposition}

\begin{proof}
For any $r>0$ from (\ref{abcd1}) we get
\begin{equation}\label{cpalr}
a(rx)+b(ry)=c(r(x+y))+d\left(\frac{x}{y}\right)\;,\;\;\;(x,y)\in(0,\infty)^2\setminus\frac{1}{r}M.
\end{equation}
Subtracting now (\ref{abcd1}) from (\ref{cpalr}) for any $r>0$ we
arrive at
\begin{equation}\label{abc}
a_r(x)+b_r(y)=c_r(x+y)\;,\;\;\;(x,y)\in(0,\infty)^2\setminus
\left(M\cup\frac{1}{r}M\right)\;,
\end{equation}
where $a_r$, $b_r$ and $c_r$ are defined by
$a_r(x)=a(rx)-a(x)$, $b_r(x)=b(rx)-b(x)$ and $c_r(x)=c(rx)-c(x)$, $x\in(0,\infty)$, respectively.

Due to measurability of $a$, $b$ and $c$ it follows from \eqref{abc} that for any
$r\in(0,\infty)$ there exist $\Lambda(r)$, $\alpha(r)$
and $\beta(r)$ such that for $l_1$-almost all $x\in(0,\infty)$
$$
a_r(x)=\Lambda(r)\:x+\alpha(r)\;,\;\;\;b_r(x)=\Lambda(r)\:x+\beta(r)\;,\;\;\;c_r(x)=\Lambda(r)\:x+\alpha(r)+\beta(r)
$$
(for more details see Theorem 6 in Section 5 below).

First, consider the functions $a_r$ for any $r>0$.

By the definition of $a_r$ and the above observation it follows
that for any $(x,y)\in(0,\infty)^2$ there exists a measurable set
$E_{xy}\subset(0,\infty)$ such that $l_1(E_{xy})=0$ and
$$
a_{xy}(z)=a(xyz)-a(z)=\Lambda(xy)\:z+\alpha(xy)\;,\;\;\;\forall\:z\in(0,\infty)\setminus
E_{xy}.
$$
That is the above identity holds on the set
$$
U_1=\{(x,y,z):\:(x,y)\in (0,\infty)^2,\:z\in(0,\infty)\setminus
E_{xy}\}.$$

Similarly
$$
a_y(xz)=a(xyz)-a(xz)=\Lambda(y)\:xz+\alpha(y)
$$
holds on
$U_2=\{(x,y,z):\:(x,y)\in(0,\infty)^2,\:z\in(0,\infty)\setminus\frac{1}{x}E_y\}
$, where for any $y>0$ the set $E_y\subset(0,\infty)$ is such that
$l_1(E_y)=0$.

Also
\begin{equation}\label{ax}
a_x(z)=a(xz)-a(z)=\Lambda(x)\:z+\alpha(x)
\end{equation}
holds on $U_3=\{(x,y,z):\:(x,y)\in(0,\infty)^2,\:z\in
(0,\infty)\setminus E_x\}$.

Taking into account the last three identities, since
$a_{xy}(z)=a_y(xz)+a_x(z)$, we arrive at
$$
\Lambda(xy)\:z+\alpha(xy)=\Lambda(y)\:xz+\alpha(y)+\Lambda(x)\:z+\alpha(x)
$$
on $V_1=U_1\cap U_2\cap
U_3=\{(x,y,z):\:(x,y)\in(0,\infty)^2,\:z\in (0,\infty)\setminus
(E_{xy}\cup\frac{1}{x}E_y\cup E_x)\}$.

Interchanging the roles of $x$ and $y$ in the above reasoning we arrive
at
$$
\Lambda(xy)\:z+\alpha(xy)=\Lambda(x)\:yz+\alpha(x)+\Lambda(y)\:z+\alpha(y)
$$
on $V_2=\{(x,y,z):\:(x,y)\in(0,\infty)^2,\:z\in
(0,\infty)\setminus (E_{xy}\cup\frac{1}{y}E_x\cup E_y)\}$.

Finally we conclude that
\begin{equation}\label{je}
\Lambda(xy)\:z+\alpha(xy)=\Lambda(y)\:xz+\alpha(y)+\Lambda(x)\:z+\alpha(x)=\Lambda(x)\:yz+\alpha(y)+\Lambda(y)\:z+\alpha(x)
\end{equation}
on $$V=V_1\cap V_2=\{(x,y,z):\:(x,y)\in(0,\infty)^2,\:z\in
(0,\infty)\setminus E_{x,y}\},
$$
where $E_{x,y}=E_{xy}\cup\frac{1}{x}E_y\cup E_x\cup
E_{xy}\cup\frac{1}{y}E_x\cup E_y$ and thus $l_1(E_{x,y})=0$.

Consequently,
$\Lambda(y)\:xz+\Lambda(x)\:z=\Lambda(x)\:yz+\Lambda(y)\:z$ for
any $x,y>0$ and any $z\in(0,\infty)\setminus E_{x,y}$. Thus taking
$x=2$ and denoting $\Lambda(2)=\lambda$ we obtain
\begin{equation}\label{ayz}
\Lambda(y)=\lambda\:(y-1)\qquad \forall\:y\in(0,\,\infty).
\end{equation}

Thus $\Lambda(xy)\:z=\Lambda(y)\:xz+\Lambda(x)\:z$ and returning to the first equation of \eqref{je} we arrive at
$\alpha(xy)=\alpha(x)+\alpha(y)$ on $(0,\infty)^2$. Note that due to \eqref{ax} and \eqref{ayz} it follows that $\alpha$ is a measurable function. Consequently,
$\alpha(x)=\varkappa\log(x)$ for $x\in(0,\infty)$, where $\varkappa=\varkappa_a$ is
a real constant.

Now, we plug \eqref{ayz} with $y$ replaced by $x$ into \eqref{ax}
getting
\begin{equation}\label{const}
a(xz)-a(z)=\lambda\:(x-1)z+\varkappa \log(x)
\end{equation}
for any $x\in(0,\infty)$ and for any $z\in(0,\infty)\setminus
(E_x\cap E_{2,x})$. Define now a new function $h:(0,\infty)\to\R$ by $h(x)=a(x)-\lambda\:x-\varkappa\log(x)$. Then (\ref{const})
has the form
$$
h(xz)=h(z)\quad\quad
\forall\:x>0\quad\mbox{and}\quad\forall\:z\in(0,\infty)\setminus
(E_x\cap E_{2,x}).
$$
By Lemma \ref{gxy=gx} below it follows that $h$ is constant, say equal
to $\alpha$, outside of a set of $l_1$ measure zero. Thus the
final formula for $a$ is proved.

The formulas for $b$ and $c$
$$
b(x)=\lambda x+\varkappa_b\log(x)+\beta\qquad\mbox{and}\qquad c(x)=\lambda x+\varkappa_c\log(x)+\gamma,
$$
which hold $l_1$-almost everywhere, follow in much the same way. Note that \eqref{abc} yields $\varkappa_a+\varkappa_b=\varkappa_c$.

To retrieve $d$ from \eqref{abcd} it suffices to change the
variables as follows $(x,y)\to(x/y,y)=(z,y)$ that is to transform
the set $(0,\infty)^2\setminus M$ by this mapping. Then by the
Fubini theorem again we conclude that there exists a set
$Z\subset(0,\infty)$ with $l_1(Z)=0$ such that for any $z\in
(0,\infty)\setminus Z$ there exists a set $E_z$ with $l_1(E_z)=0$
such that for any $z\in (0,\infty)\setminus Z$ and for any
$y\in(0,\infty)\setminus E_z$ we have
$$
d(z)=\lambda\:zy+\varkappa_a\log(zy)+\alpha+\lambda\:y+\varkappa_b\log(y)+\beta-\lambda\:(zy+y)-(\varkappa_a+\varkappa_b)\log(zy+y)-\gamma
$$$$=\varkappa_a\log(z/(z+1))+\varkappa_b\log(1/(z+1))+\alpha+\beta-\gamma.
$$
\end{proof}

\begin{lemma}\label{gxy=gx}
Let $G:(0,\infty)\to\R$ be a Borel measurable function such
that
\begin{equation}\label{cpl}
G(xy)=G(y)\;,\;\;\;\forall\:x>0\;\mbox{ and }\;\forall
y\in(0,\infty)\setminus E_x,\;\mbox{ where }\:l_1(E_x)=0.
\end{equation}

Then $G(x)$ is constant for $l_1$-almost all $x$'s.
\end{lemma}

\begin{proof}
By \eqref{cpl} we get for any $t\in\R$ and for any $x>0$ that
$$
\int_0^1\:e^{itG(xy)}\:dy=\int_0^1\:e^{itG(y)}\:dy=:w(t).
$$

Changing the variable $u=xy$ in the first integral we obtain
$$
xw(t)=\int_0^x\:e^{itG(u)}\:du.
$$
That is $\int_0^x\:\left(e^{itG(u)}-w(t)\right)\:du=0$ for any
$x>0$. Hence for any Borel set $B\subset(0,\infty)$ we have
$$
\int_B\:\left(e^{itG(u)}-w(t)\right)\:du=0.
$$
By the basic property of the Lebesgue integral we conclude that
$e^{itG(u)}=w(t)$ for every $t\in\R$ and any $u\in(0,\infty)\setminus E_t$, with $l_1(E_t)=0$.
Take arbitrary $t_1,t_2\in\R$, $t_1\ne t_2$. Then for $u\in(0,\infty)\setminus (E_{t_1}\cup E_{t_2})$ we have
$$
w(t_1)=e^{it_1G(u)}=\left(e^{it_2G(u)}\right)^{t_1/t_2}=(w(t_2))^{t_1/t_2}.
$$
Consequently, there exist constant $\varkappa\in\C$ such that $w(t)=e^{i\varkappa t}$, $t\in\R$. Finally, we conclude that $G(u)=\varkappa\in\R$ outside a set of the Lebesgue measure zero.
\end{proof}

\section{Almost everywhere Olkin-Baker functional equation without measurabi\-li\-ty}

Recently, Kominek (2011) proved that there exist solutions of \eqref{cpl} which are not constant $l_1$-almost everywhere. Actually, in that paper an additive  version of \eqref{cpl} of the form
$$
H(x+y)=H(y)\qquad x\in\R\setminus X,\;\;l_1(X)=0,\;\;y\in\R\setminus E_x,\;l_1(E_x)=0,
$$
was considered. In view of Lemma \ref{gxy=gx} these solutions are not Borel measurable.
Any function $G$ satisfying \eqref{cpl} will be called {\em  semi-constant} function.

Recall, that a function $A:(0,\infty)\to\R$ is called {\em additive} whenever it satisfies the Cauchy equation, that is $A(x+y)=A(x)+A(y)$, $x,y>0$. Similarly, a function $L:(0,\infty)\to\R$ is termed {\em a logarithmic type} function provided that $L(xy)=L(x)+L(y)$, $x,y>0$.
Semi-constant, additive and  logarithmic type functions will play the crucial role in our approach to the Olkin-Baker equation, when no regularity conditions are imposed on the unknown functions.

\begin{theorem}\label{non}
Let $a$, $b$, $c$ and $d$ be real functions on
$(0,\infty)$. Assume that there exists a measurable set
$M\subset(0,\infty)^2$ such that $l_2(M)=0$ and
\begin{equation}\label{cpl1}
a(x)+b(y)=c(x+y)+d\left(\frac{x}{y}\right)\;,\;\;\;(x,y)\in(0,\infty)^2\setminus
M.
\end{equation}
Then there exist: an additive  function
$A: (0, \infty) \longrightarrow \R$, logarithmic type functions $L_a,  L_b : (0, \infty) \longrightarrow \R$  and real constants $\alpha, \beta, \gamma$ such that
$$a(x) = A(x) + L_a(x) + \alpha, \quad b(x) = A(x) + L_b(x) + \beta,  \quad c(x) = A(x) + L_a(x) + L_b(x) + \gamma $$
and
$$ d(x) = L_a\left(\frac{x}{x+1}\right) - L_b(x+1) + \alpha + \beta - \gamma$$
for $l_1$-almost all $x\in(0,\infty)$.
\end{theorem}

\begin{proof}
Repeating the first part of the argument used in the proof of Proposition \ref{eq1} we arrive at the following representations of functions $a$, $b$ and $c$ which are valid $l_1$-almost everywhere on $(0,\infty)$
$$
a(x)=A(x)+L_a(x)+h_a(x),\quad b(x)=A(x)+L_b(x)+h_b(x),$$
$$
 c(x)=A(x)+L_a(x)+L_b(x)+h_c(x),
$$
where $h_a$, $h_b$ and $h_c$ are semi-constant functions.

Plugging these forms of unknown functions back to the original equation \eqref{cpl1}, and denoting $z=x/y$ similarly as in the previous proof we get
$$
d(z)=L_a(z/(z+1))+L_b(1/(z+1)+h_a(zy)+h_b(y)-h_c(zy+y).
$$
holding for $l_1$-almost all $z\in(0,\infty)$ and for any $y\in(0,\infty)\setminus E_z$, where $l_1(E_z)=0$. Note that, by the definition of semi-constant functions, possibly extending $E_z$ to another set $\tilde{E_z}$ but still with $l_1(\tilde{E}_z)=0$, we have
\bel{hahb}
h_a(y)+h_b(y)-h_c(y)=d(z)-L_a(z/(z+1))-L_b(1/(z+1))
\ee
which holds for $l_1$-almost all $z\in(0,\infty)$ and any $y\in (0,\infty)\setminus \tilde{E}_z$. Fix $z$ in \eqref{hahb}. Then we conclude from \eqref{hahb} that $h_a+h_b-h_c$ is $l_1$-almost everywhere constant, say $\delta$. Consequently $d(z)=L_a(z/(z+1))-L_b(1/(z+1))+\delta$.

Now, \eqref{cpl1} yields the Pexider equation
$$
h_a(x)+h_b(y)=h_c(x+y)+\delta
$$
for $l_2$-almost all $(x,y)\in(0,\infty)^2$. Consequently, by means of Theorem \ref{ideal} (see Section 5\, below), there exist an additive mapping $\tilde{A}: \R \longrightarrow \R$, real constants $\alpha, \beta$  and a set $E$ of measure zero such that
$$h_c(x) =\tilde{A}(x) + \alpha + \beta - \delta, \quad h_a(x) = \tilde{A}(x) + \alpha \quad {\rm and} \quad h_b(x) = \tilde{A}(x) + \beta $$
for all $x \in (0, \infty) \setminus E$. Since $h_a$ is semi-constant we derive the existence of a set $E_2$ of measure zero such that
$$
h_a(2y) = h_a(y) \quad {\rm for \,\, all} \quad y \in (0, \infty) \setminus E_2.
$$
Obviously $\tilde{E}:= E_2 \cup E \cup \frac{1}{2} E$ is of measure zero and, for every $y \in (0, \infty) \setminus \tilde{E}$, one has
$$   2\tilde{A}(y) + \alpha = \tilde{A}(2y) + \alpha =  h_a(2y) = h_a(y) = \tilde{A}(y) + \alpha \,.$$
This implies that\, $\tilde{A}(y) = 0$ \,for\, $l_1$-almost all positive $y$'s (actually, on account of Lemma 5, see Section 5\, below, we see that $\tilde{A}$ vanishes everywhere on $\R$). Thus
$$h_c(x) = \alpha + \beta - \delta =: \gamma, \quad h_a(x) =  \alpha \quad {\rm and} \quad h_b(x) =  \beta $$
for\, $l_1$-almost all $x \in (0, \infty)$ and the proof has been completed.

\end{proof}

%++++++++++++++++++++++++++++++++++++++++++++++++++++++++++++++++++++++++++++++++++++++++++++++++
\newpage

\section{ More abstract setting }

\vspace{1cm}

The nullsets in the preceding section may naturally be replaced by an abstract notion of
 ``negligible'' sets, i.e. members of a {\it proper
linearly invariant ideal} (briefly: p.l.i. ideal) in $\R$ defined as follows.

A nonempty family $\id\subset 2^{\R} \setminus \{\R\}$ is termed to be a
 p.l.i. ideal (resp. p.l.i. $\sigma-$ideal) provided that it is closed under finite (resp. countable) set theoretical unions, i.e.
$$ A, B \in \id \Longrightarrow A \cup B \in \id \qquad  ({\rm resp.}\,\, A_n \in \id,\, n\in \N \Longrightarrow \bigcup _{n \in \N}A_n \in \id )\,,$$
 hereditary with respect to descending inclusions, i.e.
$$ A \in \id,\ ,B \subset A \Longrightarrow B \in \id\,,$$
 and such that jointly with a given set it contains its image under any affine transformation of the real line onto itself, i.e
$$ A \in \id,\, \alpha \in \R \setminus \{0\}, \beta \in \R \Longrightarrow \alpha A + \beta \in \id\,.$$

Clearly the family of all nullsets (sets of Lebesgue measure zero) in $\R$ forms a p.l.i. $\sigma$-ideal. However there are numerous other p.l.i. ideals; let us mention
only a few of them:
\begin{itemize}
\item the family of all first category (in the sense of Baire) subsets of $\R$;
\item the family of all bounded subsets of $\R$;
\item the family of all sets of finite outer Lebesgue measure in $\R$;
\item the family of all countable subsets of $\R$;
\item given a nonempty family  ${\cal R} \subset 2^{\R} $  such that no finite union of sets of the form
$\alpha U + \beta, \,\, \alpha, \beta \in \R, \, \alpha \not = 0,\,\, U \in {\cal R},\, $ \, coincides with $\R$ the collection of all subsets of finite unions of affine images
of sets from ${\cal R}$ forms a  p.l.i. ideal (generated by ${\cal R}$).
\end{itemize}

\bigskip

\begin{remark}
 Each member of a p.l.i. $\sigma$-ideal forms a boundary set.
\end{remark}
\begin{proof}
\par If we had an interval $(\alpha, \beta)$ in a p.l.i. $\sigma$-ideal $\id$ in $\R$  then the union of  intervals  $k(-\varepsilon, \varepsilon)$ with $\varepsilon :=
\frac{1}{2}(\beta - \alpha)$, over all positive integers $k$ would coincide with the whole of $\R$, contradicting the properness of $\id$.
\end{proof}

\bigskip

We say that a property ${\cal P}(x)$ holds for
$\id$-almost all $x \in \R$
iff ${\cal P}(x)$ is valid for all $x \in \R \setminus U$ provided that $U\in\id$.
\par
For a subset $M\subset \R^2$ and $x\in \R$ we define a {\it section}
$$
M[x]:=\{y\in \R:\ (x,y)\in M\}.
$$

\par
Motivated by Fubini's Theorem we define the family
$$
\Omega(\id):=\{M\subset \R^2:\ \ M[x]\in\id\,\, \mbox{\rm for $\id$-almost all}\,\,
x \in \R \}
$$
and refer the reader to
the paper of  Ger (1975) or to the monograph of  Kuczma (2009) [Ch. XVII, \S 5]) for further details.

\medskip

In the sequel we will need the following

\begin{lemma} Given a p.l.i $\sigma$-ideal $\id$ in $\R$ and a positive number $c$, assume that an additive function $A: \R \longrightarrow \R$ enjoys the property that $A\vert_{(c, \infty) \setminus E} = 0$ for some $E \in \id.$ Then $A$ vanishes on $\R.$
\end{lemma}
\begin{proof}
\par Without loss of generality we may assume that $E = - E$. Consequently, that symmetry property is shared by the set \,$P:= \left( (- \infty, -c) \cup (c, \infty)\right) \setminus E,$ i.e.
$P = - P.$ Plainly, due to the oddness of $A$ one has \,$A\vert_{P} = 0.$ Fix arbitrarily an $x \in \R.$ We are going to show that
\begin{equation}\label{P}
P \cap (x + P) \not = \emptyset\,.
\end{equation}
Indeed, otherwise we would have
$$ P^{\prime} \cup (x + P^{\prime})  = \R, \quad {\rm where} \quad P^{\prime}: = \R \setminus P\,,$$
whence
$$ [-c, c] \cup E \cup [x-c, x+c] \cup (x+E) = \R\,.$$
In particular, we get \, int\! $(E \cup (E+x))  \not = \emptyset$, \, which contradicts Remark 1 because, obviously, the union $E \cup (E+x)$ forms a member of the $\sigma$-ideal $\id$.
\newline \indent Thus the inequality (13) has been proved. Taking now a point $p$ from the intersection $P \cap (x + P)$ we infer that both $p$ and $x-p$ belong to $P$, whence
$$A(x) = A(p) + A(x-p) = 0\,,$$
which finishes the proof.
\end{proof}

We proceed with proving the following result.

\begin{theorem}\label{ideal}
Given a p.l.i. $\sigma$-ideal  $\id$ in $\R$
assume that functions $f,g,h: (0, \infty) \longrightarrow \R$ satisfy the Pexider functional equation
\begin{equation}\label{Pexider}
 f(x+y) = g(x) + h(y)
\end{equation}
for all pairs $(x,y) \in (0, \infty)^2 \setminus M$\, and some member $M$ of the family $\Omega(\id)$
such that $T(M) \in  \Omega(\id)$ for every unimodular transformation $T$ of the real plane.
Then  there exist exactly one additive function $A: \R \longrightarrow \R$ and real constants $\alpha, \beta$ such that
$$f(x) = A(x) + \alpha + \beta, \quad g(x) = A(x) + \alpha \quad and \quad h(x) = A(x) + \beta $$
for $\id$-almost all $x \in (0, \infty)$.
\end{theorem}

\begin{proof}
\par Proceeding like in the proof of Theorem 8 from  Ger's paper (1975) we derive the existence of a postive constant $x_0$ (by means of Remark 1, being as small as required) and real constant $y_0$ such that the function
$$ F(x): = f(x+2x_0) + y_0\,, \quad x \in (0, \infty), $$
satisfies  the Cauchy functional equation $\Omega(\id)$-almost everywhere, i.e. there exists a set $N \in\Omega(\id)$ such that
$$ F(x+y) = F(x) + F(y) \quad {\rm for\,\, all \,\, pairs} \quad  (x,y) \in (0,\infty)^2 \setminus N\,.$$
Moreover, we have also
$$ g(x) = f(x+x_0) + y_1 \quad {\rm and} \quad h(x) = f(x+x_0) + y_2 \qquad  {\rm for}\,\,\,  \id\!\!-\!\!{\rm almost \,\, all} \quad x \in (0, \infty),$$
with some real constants $y_1, y_2$.
It is not hard to check (somewhat tedious but easy calculations using the unimodular images of $M$) that then the function $\Phi: \R \longrightarrow \R$ given by the formula
\begin{eqnarray*}
\Phi(x) : = \left\{\begin{array}{lll}
            F(x) & \mbox{whenever $x \in (0, \infty)$ }\\
            0 & \mbox{for $x = 0$}\\
          - F(-x) & \mbox{whenever $x \in (- \infty, 0)$ }
\end{array} \right.
\end{eqnarray*}
admits a member $N_0$ of the family $\Omega(\id)$ such that
$$ \Phi (x+y) =\Phi (x) + \Phi (y) \quad {\rm for\,\, all \,\, pairs} \quad  (x,y) \in \R^2 \setminus N_0\,.$$
An appeal to the main result of de Bruijn (1966) (see also  Ger (1978) where the notation coincides with that used in the present paper)  gives the existence of exactly one additive function $A: \R \longrightarrow \R$  such that
$$\Phi (x) = A(x) \, \quad  {\rm for}\,\,\,  \id\!\!-\!\!{\rm almost \,\, all} \,\,  x \in \R\,.$$
Consequently, there exists a set $E(x_0) \in  \id $ and some real constants $\alpha_0, \beta_0$ such that
$$ f(x + 2x_0) = A(x) + \alpha_0 + \beta_0, \quad g(x+ x_0) = A(x) + \beta_0 \quad {\rm and} \quad h(x+x_0) = A(x) + \beta_0$$
for all $x \in (0, \infty) \setminus E(x_0)$. Hence
$$ f(t) = A(t) + \alpha + \beta \quad {\rm for} \quad t \in (2x_0, \infty) \setminus \tilde{ E}(x_0), \quad g(t) = A(t) + \beta \quad {\rm for} \quad t \in (x_0, \infty) \setminus \tilde{ E}(x_0)$$
$${\rm and} \quad h(t) = A(t) + \beta \quad {\rm for} \quad t \in (x_0, \infty) \setminus \tilde{ E}(x_0),$$
where we have put $ \tilde{ E}(x_0):=( E(x_0) + 2x_0) \cup ( E(x_0) + x_0) \in \id$\, and \, $\alpha: =  \alpha_0 - A(x_0), \\ \beta: = \beta_0 - A(x_0).$
Since, as it was told earlier, the point $x_0$ might be chosen as small as we wish, for every positive integer $n$ there exist:\, an additive map $A_n: \R \longrightarrow \R$, a set $E_n \in \id $ \, and  real constants $\alpha_n, \beta_n$\, such that  for each $ t \in (\frac{1}{n}, \infty) \setminus E_n$ one has
$$ f(t) = A_n(t) + \alpha_n+ \beta_n,  \quad  g(t) = A_n(t) + \alpha_n \quad {\rm and} \quad h(t) = A_n(t) + \beta_n\,.$$
Fix arbitrarily positive integers $n, m, \, n < m,$\, to get
$$ A_n(t) + \alpha_n =  A_m(t) = \alpha_m \quad {\rm for\,\, all} \quad t \in  (\frac{1}{n}, \infty) \setminus E\,,$$
where we have put
$$ E: = \bigcup_{n=1}^{\infty} E_n \in \id \,.$$
Now, fix a $t \in (\frac{1}{n}, \infty) \setminus \tilde{E}$ \, with \, $\tilde{E}:= \bigcup_{k=1}^{\infty} \frac{1}{k} E \in \id \,;$ then $kt \in (\frac{1}{n}, \infty) \setminus E$\, for every positive integer $k$,\, whence
$$ k A_n(t) + \alpha_ n = A_n(kt) + \alpha_n = A_m(kt) + \alpha_m = kA_m(t) + \alpha_m$$
and, a fortiori,
$$A_n(t) = A_m(t) \quad {\rm for\,\, all} \quad t \in  (\frac{1}{n}, \infty) \setminus \tilde{E}\,,$$
Therefore, the additive function $A_n - A_m$ vanishes $\id$-almost everywhere on the halfline $ (\frac{1}{n}, \infty)$ whence in view of Lemma 5,  $A_n = A_m = : A$ does not depend on $n$ as well as the constants $\alpha_n = : \alpha$ and $\beta_n =: \beta.$
This forces the equalities
$$ f(t) = A(t) + \alpha+ \beta,  \quad  g(t) = A(t) + \alpha \quad {\rm and} \quad h(t) = A(t) + \beta\,$$
to be valid for every $ t \in (\frac{1}{n}, \infty) \setminus E, \, n \in \N,$
which completes the proof.
\end{proof}

A careful inspection of the proof of Theorem \ref{non} \,as well as that of Proposition \ref{eq1} ensures that dealing with the p.l.i. $\sigma$-ideal of all sets of Lebesgue measure zero in $\R$ we were using exclusively these properties of that set family which are axiomatically guaranteed in the definition of an abstract p.l.i. ideal. Therefore we terminate this paper with the statement of the following generalization of Theorem \ref{non}.

\begin{theorem}\label{last}
Given a p.l.i. $\sigma$-ideal  $\id$ in $\R$
assume that functions $a,b,c, d: (0, \infty) \longrightarrow \R$ satisfy the functional equation {\rm (11)}
for some member $M$ of the family $\Omega(\id)$.
Then in the case where the set  $T(M)$ falls into $ \Omega(\id)$ for every unimodular transformation $T$ of the real plane, there exist: an additive  function
$A: (0, \infty) \longrightarrow \R$, logarithmic type functions $L_a,  L_b : (0, \infty) \longrightarrow \R$  and real constants $\alpha, \beta, \gamma$
 such that
$$a(x) = A(x) + L_a(x) + \alpha, \quad b(x) = A(x) + L_b(x) + \beta,  \quad c(x) = A(x) + L_a(x) + L_b(x) + \gamma $$
and
$$ d(x) = L_a\left(\frac{x}{x+1}\right) - L_b(x+1) + \alpha + \beta - \gamma$$
for $\id$-almost all $x \in (0, \infty)$.
\end{theorem}

\bigskip

Noteworthy seems to be the following final

\begin{remark}
Instead of all possible unimodular transformations of the plane, spoken of in both Theorem \ref{ideal}\, and\,  Theorem \ref{last}, \,it would suffice to consider only three specific ones:\,  $T_1(x,y) = (y,x), T_2(x,y) = (x+y, -y),\, T_3(x,y) = (-x-y, x), \, (x,y) \in \R^2.$
\end{remark}

%+++++++++++++++++++++++++++++++++++++++++++++++++++++++++++++++++++++++++++++++++++++++++++++++++
\small
{\bf References}
\begin{enumerate}

\item {\sc Baker, J.A.}, On the functional equation
$f(x)g(y)=p(x+y)q(x/y)$. {\em Aeq. Math.} {\bf 14}  (1976), 493-506.

\item {\sc Bobecka, K., Weso\l owski, J.}, The
Lukacs-Olkin-Rubin theorem without invariance of the "quotient".
{\em Studia Math.} {\bf 152(2)}  (2002), 147-160.

\item {\sc Bobecka, K., Weso\l owski, J.},  Multivariate
Lukacs theorem. {\em J. Multivar. Anal.} {\bf 91}  (2004),
143-160.

\item {\sc Boutouria, I.}, Characterization of the Wishart
distributions on homogeneous cones. {\em Com. Ren. Math.} {\bf 341
(1)}  (2004), 43-48.

\item {\sc Boutouria, I.}, Characterization of the Wishart
distributions on homogeneous cones in the Bobecka and Wesolowski
way. {\em Comm. Statist. Th. Meth.} {\bf 13(15)} (2004),
2552-2566.

\item {\sc Boutoria, I., Hassairi, A., Massam, H.}, Extension of the Olkin and Rubin characterization of the Wishart dsitribution on homogeneous cones.
{\em Inf. Dim. Anal. Quant. Probab. Rel. Top.} {\bf 14(4)} (2011), 591-611.

\item {\sc Casalis, M., Letac, G.}, The Lukacs-Olkin-Rubin
characterization of Wishart distribtuions on symmetric cones. {\em
Ann. Statist.} {\bf 24} (1996), 763-786.

\item {\sc de Bruijn, N.G.}, On almost additive functions,
{\em Colloq. Math.} {\bf 15} (1966), 59-63.

\item {\sc Dar\'oczy, Z., Lajk\'o, K., Sz\'ekelyhidi, L.},
Functional equations on ordered fields. {\em Publ. Math. Debrecen}
{\bf 24}  (1977), 173-179.

\item {\sc Ger, R.},  On some functional equations with a
restricted domain. {\em Fund. Math.} {\bf 89} (1975), 131-149.

\item {\sc Ger, R.},  Note on almost additive functions.
 {\em Aeq. Math.} {\bf 17} (1978), 73-76.

\item {\sc Hassairi, A., Lajmi, S, Zine, R.}, A characterization of
the Riesz distribution. {\em J. Theor. Probab.} {\bf 21(4)} (2008), 773-790.

\item {\sc Ko\l odziejek, B.}, {\em The Wishart distribution on the Lorentz cone}, Fac. Math. Infor. Sci., Warsaw Univ. Tech.. MSc Thesis, Warsaw, 2010  - in Polish.

\item {\sc Kominek, Z.}, On non-measurable subsets of the real line. {\em Folia Math.} {\bf 67} (2011), 63-66.

\item {\sc Kuczma, M.}, {\it An Introduction to the Theory of Functional
Equations and Inequalities}, Birkh${\rm \ddot{a}}$user Verlag, Basel-Boston-Berlin, 2009.

\item {\sc Lajk\'o, K.}, Remark to a paper by J. A. Baker.
{\em Aeq. Math.} {\bf 19}  (1979), 227-231.

\item {\sc Lajk\'o, K., M\'esz\'aros, F.}, Multiplicative type functional equations arising from characterization problems. {\em Aeq. Math.}  (2012), DOI: 10.1007/s00010-012-0117-2.

\item {\sc Lukacs, E.}, A characterization of the gamma
distribution. {\em Ann. Math. Statist.} {\bf 26} (1955), 319-324.

\item {\sc M\'esz\'aros, F.}, A functional equation and its application to the characterization of gamma distribution. {\em Aeq. Math.} {\bf 79} (2010), 53-59.

\item {\sc Olkin, I.}, Problem (P128). {\em Aeq. Math.} {\bf 12} (1975), 290-292.

\item {\sc Olkin, I., Rubin, H.}, A characterization of the
Wishart distribution. {\em Ann. Math. Statis.} {\bf 33} (1962),
1272-1280.

\item {\sc Weso\l owski, J.}, Multiplicative Cauchy functional equation and the equation of proportion on the Lorentz
cone. {\em Studia Math.} {\bf 179(3)} (2007), 263-275.

\end{enumerate}
\end{document}